\documentclass[microtype,12pt]{article}
\usepackage[utf8x]{inputenc}
\usepackage{amsmath,amsthm,amssymb,mathrsfs}
\usepackage{enumerate}
\usepackage{a4}
\usepackage[T1]{fontenc}
\usepackage{ae,aecompl}
\usepackage{times}
\usepackage{slashed}
\usepackage[english]{babel}
\usepackage{fancyhdr}
\usepackage{braket}
\usepackage[pagebackref,hypertexnames=false]{hyperref}
\usepackage[alphabetic,backrefs]{amsrefs}

\theoremstyle{plain}
\newtheorem{thm}{Theorem}[section]
\newtheorem{cor}[thm]{Corollary}

\newtheorem{prop}[thm]{Proposition}
\theoremstyle{definition}

\newtheorem{rem}[thm]{Remark}
\newtheorem{ex}[thm]{Example}

\newcounter{thehypothesis}
{\refstepcounter{thehypothesis}\begin{itemize}\item[(H\arabic{thehypothesis})]}
{\end{itemize}}

\newcounter{theproperty}
{\refstepcounter{theproperty}\begin{itemize}\item[(P\arabic{theproperty})]}
{\end{itemize}}

\newcommand{\om}{\omega}
\newcommand{\Om}{\Omega}
\renewcommand{\a}{\alpha}
\renewcommand{\b}{\beta}

\newcommand{\im}[1]{{\mathrm{Im}\,  #1  }}
\newcommand{\re}[1]{{\mathrm{Re}\,  #1  }}

\newcommand{\Comment}[2][\empty]{\ifthenelse{\equal{#1}{\empty}}{\todo[color=gray!10]{#2}}{\todo[color=gray!10,#1]{#2}}}

\author{Andriy Haydys\\
        University of Bielefeld
}
\title{Isolated singularities of affine special K\"ahler metrics in two dimensions}
\date{May 3, 2015}

\begin{document}
\maketitle

 \begin{abstract}
 We prove that there are just two types of isolated singularities of special K\"ahler metrics in real dimension two provided the associated holomorphic cubic form does not have essential singularities. 
 We also construct examples of such metrics.  
 \end{abstract}

\section{Introduction}

Special K\"ahler metrics attracted a lot of interest recently both in mathematics and physics, see for instance~\cite{ AlekseevskyEtAl15_QKmetrics, MaciaSwann15_TwistGeom, Neitzke14_NotesNewConstrOfHK} for the most recent works. 
The main source of interest to such metrics is the fact that the total space of the cotangent bundle of the underlying manifold carries a natural hyperK\"ahler metric for which each fiber is a Lagrangian submanifold. 
Such manifolds play an important role in the SYZ-conjecture~\cite{SYZ96_SYZConj}.

Soon after special K\"ahler metrics entered the mathematical scene~\cite{Freed99_SpecialKaehler}, Lu proved~\cite{Lu99_NoteOnSpKaehler} that  there are no complete special K\"ahler metrics besides flat ones. 
This motivates studying singular special K\"ahler metrics as the natural structure on bases of holomorphic Lagrangian fibrations with singular fibers.    
In this paper we study isolated singularities of affine special K\"ahler metrics in the lowest possible dimension.

Recall that a K\"ahler manifold $(M, g, I, \om)$ is called (affine) special K\"ahler, if it is equipped with a symplectic, torsion-free, flat connection $\nabla$ such that 
\begin{equation}\label{Eq_SpecKaehlerCond}
(\nabla_XI)Y=(\nabla_Y I)X
\end{equation}  
for all tangent vectors $X$ and $Y$.
To any special K\"ahler metric one can associate a holomorphic cubic form $\Xi$, which measures the difference between the Levi-Civita connection and $\nabla$. 

Throughout the rest of this paper we assume that $\dim_{\mathbb R}M =2$, i.e. $M$ is a Riemann surface.
Let $m_0$ be an isolated singularity of $g$.
Denote by $n$ the order of $\Xi$ at $m_0$, i.e. $m_0$ is a zero of order $n$ if $n>0$ or $m_0$ is a pole of order $|n|$ if $n<0$ or $\Xi(m_0)$ exists and does not vanish if $n=0$.    
By choosing a holomorphic coordinate $z$ near $m_0$, we can assume that $g$ is a special K\"ahler metric on the punctured disc $B_1^*=B_1(0)\setminus\set{0}$. 
The following is the main result of this paper.
\begin{thm}\label{Thm_SingOfSpecialKaehler}
Let $g=w|dz|^2$ be a special K\"ahler metric on $B_1^*$. 
Assume that $\Xi$ is holomorhic on the punctured disc and the order of $\Xi$ at the origin is $n>-\infty$.  
Then 
\begin{equation}\label{Eq_SingOfsKmetrics}
w=-|z|^{n+1}\log |z| e^{O(1)}\qquad\text{or}\qquad w=|z|^\beta(C+o(1))
\end{equation}
as $z\to 0$, where $C>0$ and $\beta<n+1$.

Moreover, for any $n\in\mathbb Z$ and $\b\in\mathbb R$ such that $\b<n+1$ there is an affine special K\"ahler metric satisfying~\eqref{Eq_SingOfsKmetrics}.
\end{thm}
     
We prove this theorem by establishing an intimate relation between special K\"ahler metrics and metrics of non-positive Gaussian curvature and applying the machinery developed for the latter ones. 
This relation allows us in particular to construct explicit examples of special K\"ahler metrics from metrics of \emph{constant} negative curvature. 

\paragraph{Acknowledgements.}

I would like to thank V.~Cort\'{e}s and R.~Mazzeo for helpful discussions and German Research Foundation (DFG) for financial support via CRC 701.

\section{Affine special K\"ahler metrics in local coordinates}

Let $\Omega\subset \mathbb C$ be an open subset, which we equip with the flat metric $g_0=|dz|^2=dx^2+dy^2$. 
Denote by $*$ the corresponding Hodge operator and by $\Delta=\partial_{xx}^2 +\partial_{yy}^2$ the Laplace operator.

\begin{prop}
For any pair $(u,\eta)\in C^\infty(\Om)\times \Om^1(\Om)$ satisfying
 \begin{align}
 d\eta &= 0,\label{Eq_sKSystemIff_deta}\\
 *d*\eta &= 2*(*\eta\wedge du) - 2e^u|\eta|^2,\label{Eq_sKSystemIff_d*eta}\\
 \Delta u &= |2\eta + e^{-u}du|^2e^{2u}.\label{Eq_sKSystemIff_Deltau}
 \end{align}
the metric $g=e^{-u}(dx^2+ dy^2)$ is special K\"ahler. 
Conversely, for any special K\"ahler metric on $\Om$ there exists a solution $(u,\eta)$ of~\eqref{Eq_sKSystemIff_deta}-\eqref{Eq_sKSystemIff_Deltau}.  
\end{prop}
\begin{proof} 
In real coordinates $(x,y)$ the connection $\nabla$ can be represented by
\[
\omega_\nabla=
\begin{pmatrix}
\om_{11} & \om_{12}\\
\om_{21} & \om_{22}
\end{pmatrix}
\in \Omega^1(\mathbb R^2; \mathfrak{gl}_2(\mathbb R)).
\]
Then~\eqref{Eq_SpecKaehlerCond} can be written as $[\om(X), I_0](Y)=[\om (Y), I_0]X$, where 
\[I_0= \begin{pmatrix}
0 & -1\\
1 & 0\end{pmatrix}.
\] 
Since this equation is symmetric with respect to $X$ and $Y$, it is enough to check its validity for $(X,Y)=(\frac{\partial}{\partial x}, \frac{\partial}{\partial y})$, which yields
\begin{equation}\label{Eq_SKincoord}
 \begin{aligned}
 &\om_{11} (\tfrac{\partial}{\partial x}) - \om_{22} (\tfrac{\partial}{\partial x}) = 
 -\bigl ( \om_{12} (\tfrac{\partial}{\partial y}) + \om_{21} (\tfrac{\partial}{\partial y}) \bigr),\\
 &\om_{12} (\tfrac{\partial}{\partial x}) + \om_{21} (\tfrac{\partial}{\partial x})=
 \om_{11} (\tfrac{\partial}{\partial y}) - \om_{22} (\tfrac{\partial}{\partial y}).
 \end{aligned}
\end{equation}
Write $\omega_{11}-\omega_{22}=a\, dx+b\, dy$. 
It follows from~\eqref{Eq_SKincoord} that $\omega_{12}+\omega_{21}=b\, dx-a\, dy$. 
Denoting $\omega_{11}+\omega_{22}=p\, dx + q\, dy$, we obtain
\begin{equation}\label{Eq_AuxOm11AndOm22}
\omega_{11} =\frac{p+a}2dx+\frac{q+b}2 dy\quad\text{and}\quad 
\omega_{22} =\frac{p-a}2dx+\frac{q-b}2 dy.
\end{equation}

Furthermore, the torsion of $\nabla$ vanishes if and only if 
$\omega_{12}(\frac \partial{\partial x})=\omega_{11}(\tfrac \partial{\partial y})$ and 
$\omega_{21}(\frac \partial{\partial y})=\omega_{22}(\tfrac \partial{\partial x})$. 
Combing this with~\eqref{Eq_AuxOm11AndOm22} we obtain
\[
\omega_{12} =\frac{q+b}2dx-\frac{p+a}2 dy\quad\text{and}\quad 
\omega_{21} =\frac{b-q}2dx+\frac{p-a}2 dy,
\]
which yields
\[
\omega_\nabla =
 \begin{pmatrix}
 \om_{11} & -*\om_{11}\\
 *\om_{22} & \om_{22}
 \end{pmatrix}.
\]   
Then $\nabla$ is flat if and only if 
\begin{equation}\label{Eq_NablaIsFlat}
 \begin{aligned}
 & d\om_{11} = \om_{11}\wedge\om_{22}, &\qquad & d*\om_{11} = -*\om_{11}\wedge\om_{22} -|\om_{11}|^2 dx\wedge dy,\\
 & d\om_{22} = \om_{22}\wedge\om_{11}, & & d*\om_{22} = -*\om_{22}\wedge\om_{11} -|\om_{22}|^2 dx\wedge dy.
 \end{aligned}
\end{equation}
Here we used a special property of 1-forms in dimension 2, namely the identity $(*\a)\wedge (*\b)=\a\wedge\beta$. 

Furthermore, notice that $\nabla$ preserves the K\"ahler form $\om= 2e^{-u} dx\wedge dy$ if and only if
\begin{equation*} 
d e^{-u}= d\bigl (\om (\tfrac \partial{\partial x}, \tfrac \partial{\partial y})\bigr )=
e^{-u}(\om_{11}+\om_{22})\quad\Leftrightarrow\quad
-du=\om_{11}+\om_{22}.
\end{equation*}
Substituting this in~\eqref{Eq_NablaIsFlat} we obtain 
\begin{equation*} 
 \begin{aligned}
 & d\om_{11} = du\wedge\om_{11}, \\ 
 & d*\om_{11} = *\om_{11}\wedge du -2|\om_{11}|^2 dx\wedge dy,\\
 & \Delta u = -4*(*\om_{11}\wedge du )+4|\om_{11}|^2 + |du|^2= |2\om_{11} + du|^2.
 \end{aligned}
\end{equation*} 
Finally, substitute $\eta = e^{-u}\om_{11}$ to obtain~\eqref{Eq_sKSystemIff_deta}-\eqref{Eq_sKSystemIff_Deltau}.
\end{proof}

Observe that~\eqref{Eq_sKSystemIff_Deltau} implies that the Gaussian curvature of $\tilde g = e^{2u} |dz|^2$ equals $-|2\eta + e^{-u}du|^2$ and therefore is non-positive. 
This relation between special K\"ahler metrics and metrics of non-positive Gaussian curvature is central for the rest of the paper and will be more vivid below.

\begin{cor}\label{Cor_sKOnSimplyConn}
Any pair of functions $(h,u)$ satisfying 
\begin{equation}\label{Eq_sKSystemFor2Funct_hu}
\Delta h=0\quad\text{and}\quad \Delta u=|dh|^2e^{2u},
\end{equation}
on a domain $\Om\subset \mathbb R^2$ determines a special K\"ahler metric on $\Om$. 
Conversely, if $H^1(\Om;\mathbb R)$ is trivial, then any special K\"ahler metric on $\Om$ determines a solution of~\eqref{Eq_sKSystemFor2Funct_hu}.  
\end{cor}
\begin{proof}
Assume $\eta=df$, which is always the case provided $H^1(\Om;\mathbb R)$ is trivial. 
By~\eqref{Eq_sKSystemIff_d*eta} we obtain
\begin{equation*}
\Delta f = 2*(*df\wedge du) - 2e^u|df|^2.
\end{equation*}
Denoting $h=2f - e^{-u}$, we compute:
\begin{equation*}
\begin{aligned}
\Delta h &= 2\Delta f - e^{-u}(-\Delta u + |du|^2)\\
 &= 4*(*df\wedge du) - 4e^u|df|^2 
    -e^{-u}\bigl ( -4e^{2u}|df|^2 +4e^u*(*df\wedge du) \bigr )\\
 &=0.
\end{aligned}
\end{equation*}
Hence, $(h,u)$ solves~\eqref{Eq_sKSystemFor2Funct_hu}.

Conversely, for any solution  $(h,u)$ of~\eqref{Eq_sKSystemFor2Funct_hu} the pair $(u, df)$ solves~\eqref{Eq_sKSystemIff_deta}-\eqref{Eq_sKSystemIff_Deltau}, where $f=(h+e^{-u})/2$. 
\end{proof}


Observe that if $h$ is a constant function, then~\eqref{Eq_sKSystemFor2Funct_hu} reduces to $\Delta u =0$, i.e., the corresponding special K\"ahler metric $g=e^{-u} (dx^2+dy^2)$ is flat. 
Non-trivial examples will be constructed below. 

%

\begin{cor}
Assume $\Om=B_1^*$. 
Then any triple $(h,u, a)$ satisfying
\begin{equation}\label{Eq_sKonPunctBall}
\Delta h=0\quad\text{and}\quad \Delta u =|dh +a\varphi|^2e^{2u},
\end{equation}
where $a\in\mathbb R$ and $\varphi$ is a 1-form generating $H^1(B_1^*;\mathbb R)$, determines a special K\"ahler metric on the punctured disc. 
Conversely, any special K\"ahler metric on the punctured disk determines a solution of~\eqref{Eq_sKonPunctBall}.
\end{cor}
\begin{proof}
Recall that  
\[
\varphi = \frac {ydx -xdy}{x^2+y^2} 
\]
is harmonic 1-form generating $H^1(B_1^*;\mathbb R)$. 
Hence, any closed $\eta\in\Om(B_1^*)$ can be written in the form $\eta = df + \tfrac a2\varphi$ for some $a\in\mathbb R$. 
Just like in the proof of Corollary~\ref{Cor_sKOnSimplyConn} put $h=2f-e^{-u}$ to obtain~\eqref{Eq_sKonPunctBall}.   
\end{proof}

Notice that the second equation of~\eqref{Eq_sKonPunctBall} (as well as~\eqref{Eq_sKSystemFor2Funct_hu}) is the celebrated Kazdan--Warner equation~\cite{KazdanWarner74_CurvFunctOnComp}.

\begin{rem}
Tracing through the proof one easily sees that given a solution of~\eqref{Eq_sKonPunctBall} the corresponding special K\"ahler structure is given by 
\begin{equation}\label{Eq_SKstructExpl}
\begin{aligned}
g &=e^{-u}(dx^2+dy^2), \qquad 
\om_\nabla = \begin{pmatrix}
 \om_{11} & -*\om_{11}\\
 *\om_{22} & \om_{22}
 \end{pmatrix},\\ 
 \om_{11} &= \frac {e^u}2 ( dh +a\varphi) - \frac 12 du,\quad \om_{22}= -\frac {e^u}2 (dh +a\varphi) -\frac 12 du.
\end{aligned}
\end{equation}
\end{rem}
\begin{ex}
Let $h$ be a positive harmonic function. 
It is straightforward to check that the pair $(h,-\log h)$ solves~\eqref{Eq_sKSystemFor2Funct_hu}.
Choosing $h=-\log |z|$ we obtain that $g = -\log |z|\, |dz|^2$ is a special K\"ahler metric on the punctured unit disc.
Special K\"ahler metrics with logarithmic singularities were studied for instance in~\cite{GrossWilson00_LargeCxStrLimits, Loftin05_SingularSemiFlat}. 
\end{ex}
\begin{ex}\label{Ex_MainExOfSKmetrics}
Pick any integer $n$ and consider the harmonic function
\[
h(x,y)=
\begin{cases}
\re z^{n+1} & \text{if } n\neq -1,\\
\log |z| & \text{if } n=-1.
\end{cases}
\]
Clearly, there are some positive constants $C_1$ and $C_2$ such that $-C_1|z|^{2n}\le -|dh|^2\le -C_2 |z|^{2n}$.
Choose a point $z_0\in B_2(0)\setminus B_1(0)$ and a non-positive smooth function $\tilde K$ satisfying
\begin{align*}
 &\tilde K |_{B_1(0)}= -|dh|^2,\qquad \tilde K |_{\mathbb C\setminus B_2(0)} = -1,\\
 -& C_1'|z-z_0|^{1-2n}\le \tilde K\le -C_2'|z-z_0|^{1-2n},
\end{align*}
where $C_1'$ and  $C_2'$ are some positive constants. 
Clearly, $\tilde K$ can be extended as a smooth function to $\mathbb{CP}^1\setminus\set{z_0, 0}$, where we think of $\mathbb{CP}^1$ as $\mathbb C\cup\set{\infty}$.
 
Let $g_0$ be a Riemannian metric on $\mathbb{CP}^1$ such that $g_0=|dz|^2$ on $B_1(0)$. 
Denote by $K_0$ the Gaussian curvature of $g_0$.
By~\cite{McOwen93_PrescribedCurvature}*{Thm.II} for any $\beta\le n+1$ there exists a solution $u$ of 
$\Delta_{g_0} u +\tilde Ke^{2u}= K_0$ such that 
\[
u=
\begin{cases}
-\b \log |z| + c + o(1), &\text{if } \b<n+1,\\
-(n+1)\log |z| -\log\bigl |\log |z| \bigr | + O(1) &\text{if } \b = n+1,
\end{cases}
\qquad \text{as } z\to 0.
\] 
($u$ has also a similar behaviuor near $z_0$.)
Thus, on $B_1(0)$ the pair $(h,u)$ solves~\eqref{Eq_sKSystemFor2Funct_hu}. 
In particular, $g=w|dz|^2=e^{-u}|dz|^2$ is a special K\"ahler metric satisfying~\eqref{Eq_SingOfsKmetrics}.
\end{ex}

\begin{proof}[Proof of Theorem~\ref{Thm_SingOfSpecialKaehler}] 
Let $\pi^{(1,0)}\in\Omega^{1,0}(T_{\mathbb C}B_1^*)$ be the projection onto $T^{1,0}B_1^*$. 
Recall the definition of the holomorphic cubic form: 
\[\Xi= -\om \bigl (\pi^{(1,0)}, \nabla \pi^{(1,0)} \bigr ).
\]
Then, a direct computation yields 
\[
\Xi= \frac {e^{-u}}4\bigl ( *(\omega_{11}-\omega_{22}) -i(\omega_{11}-\omega_{22}) \bigr )dz^2.
\]
Therefore, substituting~\eqref{Eq_SKstructExpl} we obtain $\Xi =\Xi_0dz^3$, where
\begin{align*}
\Xi_0=\frac 12 \Bigl ( \frac a{2z}-\frac{\partial h}{\partial z}i \Bigr).
%
%
\end{align*} 
Notice that $16 |\Xi_0|^2=|dh+a\varphi|^2=:-\tilde K$.

Denoting $w=e^{-u}$, we obtain $\Delta u=- \tilde K e^{2u}$ on the punctured disc. 
 Under the hypothesis of this theorem there exist positive constants $C_1$ and $C_2$ such that the inequalities
\begin{equation}\label{Eq_EstimOnCurvature}
-C_1|z|^{2 n}\le \tilde K\le -C_2|z|^{2n}
\end{equation}
hold on a (possibly smaller) punctured disc.


By~\cite{McOwen93_PrescribedCurvature}*{Appendix~B} we obtain that
\[
u=-\b\log |z| +c +o(1)\qquad\text{or}\qquad u=-(n+1)\log |z|-\log |\log |z||  +O(1),
\]
where $\b<  n +1$.
Since $w=e^{-u}$, we obtain~\eqref{Eq_SingOfsKmetrics}.

The existence of special K\"ahler metrics satisfying~\eqref{Eq_SingOfsKmetrics} has been established in Example~\ref{Ex_MainExOfSKmetrics}.
\end{proof}

We remark in passing that starting from a different perspective, Loftin~\cite{Loftin05_SingularSemiFlat} utilized an equation equivalent to~\eqref{Eq_sKSystemIff_Deltau} to study special K\"ahler metrics on $\mathbb{CP}^1$.

\section{Metrics of constant negative curvature and further examples}

In this section we construct more examples --- in particular explicit --- of special K\"ahler metrics from metrics of \emph{constant} negative Gaussian curvature.

\begin{prop}\label{Prop_ConstNegCurv_sK}
Let $\tilde g=e^{2u}|dz|^2$ be a metric of constant negative Gaussian curvature on a domain $\Om\subset \mathbb R^2$. 
Then $g=e^{-u}|dz|^2$ is special K\"ahler.
\end{prop}
\begin{proof}
Since $\tilde g$ has a constant negative scalar curvature, say $-1$, we have $\Delta u = e^{2u}$. 
Put $h(x,y)=x$ and observe that the pair $(h,u)$ satisfies~\eqref{Eq_sKSystemFor2Funct_hu}. 
Hence, $g=e^{-u}|dz|^2$ is a special K\"ahler metric. 
\end{proof}
\begin{ex}\label{Ex_PoincareMetric}
Different incarnations of the Poicar\'{e} metric lead to differently looking special K\"ahler metrics, which are gathered in the following table:
\begin{table}[h]
\renewcommand{\arraystretch}{1.3}
\begin{tabular}{|c|c|c|}
\hline Constant negative curvature metric & Domain &  Special K\"ahler metric\\ 
\hline  $\tilde g=(\im z)^{-2}|dz|^2$ & upper half-plane & $g=\im z|dz|^2$ \\ 
\hline $\tilde g= 4(1-|z|^2)^{-2} |dz|^2$ & unit disc & $g= \frac 12(1-|z|^2)|dz|^2$ \\ 
\hline $\tilde g = (|z|\log |z|)^{-2}|dz^2|$ & punctured disc & $ g = -|z|\log |z|\,|dz|^2$ \\ 
\hline 
\end{tabular}
\caption{Poincar\'{e} metrics and corresponding special K\"ahler metrics}
\end{table} 

The special K\"ahler metric appearing in the first row can be found in \cite{Freed99_SpecialKaehler}*{Rem.\,1.20}. 
Particularly interesting for us is the one appearing in the last row, as this yields an example of special K\"ahler metric with an isolated singularity.
\end{ex}
\begin{ex}
Explicit local examples of metrics of constant negative Gaussian curvature with conical singularities can be found in~\cite{KrausRoth08_BehaviurOfSolutions}*{Ex.2.1}. 
For instance, if $\alpha\in(0,1)$, then 
\[
\tilde g= \frac 1{|z|^{2\alpha}}\Bigl (\frac {1-\a}{1-|z|^{2(1-\alpha)}}\Bigr)^2|dz|^2
\]
is such a metric. Hence, 
\[
g = \frac 1{1-\a}   |z|^\alpha  \bigl ({1-|z|^{2(1-\alpha)}}\bigr )\, |dz|^2
\]
is a special K\"ahler metric with conical singularity.
\end{ex}
\begin{ex}\label{Ex_PicardsMetrics}
It is a classical result of Picard~\cite{Picard1893_Delequation} that for any given $n\ge 3$  pairwise distinct points $(z_1, \dots, z_n)$ in $\mathbb C$ and any $n$ real numbers $(\alpha_1, \dots, \alpha_n)$ such that $\alpha_j<1$ and $\sum\alpha_j>2$ there exists a metric of constant negative curvature $\tilde g$ on $\mathbb C\setminus\{ z_1,\dots, z_n \}$ satisfying $\tilde g=|z-z_j|^{-2\alpha_j}(c+o(1))|dz|^2$ near  $z_j$.  
Hence, the corresponding special K\"ahler metric $g$ has a conical singularity near $z_j$:
\[
g=|z-z_j|^{\alpha_j}(c+o(1)) |dz|^2.
\]

Explicit examples of constant negative curvature --- hence special K\"ahler --- metrics on three times punctured complex plane can be found in~\cite{KrausRothSugawa11_MetricsWithConicalSing} and references therein. 

Notice also that it is possible to allow $\alpha_i=1$ changing the asymptotic behaviour correspondingly and to replace $\mathbb C$ by $\mathbb P^1$.
\end{ex}

\bibliography{literatura}

\end{document}